\documentclass[11pt]{amsart}
\usepackage{amsmath,amsthm,amsfonts,amssymb,verbatim,eucal}
\usepackage[all]{xy}
\usepackage{extarrows}
\numberwithin{equation}{section}

\newtheorem{thm}[equation]{Theorem} 

\newtheorem{lemma}[equation]{Lemma} 

\newtheorem{example}[equation]{Example}
\newtheorem{remark}[equation]{Remark}

\DeclareMathOperator{\gr}{gr}

\newcommand{\DOT}{\setlength{\unitlength}{1pt}\begin{picture}(2.5,2)
               (1,1)\put(2.5,2.5){\circle*{2}}\end{picture}}
\newcommand{\bu}{\DOT}

\newcommand{\ld}{\lambda}

\newcommand{\Hom}{\mbox{\rm Hom\,}}

\newcommand{\ot}{\otimes}

\newcommand{\cH}{\mathcal{H}}

\DeclareMathOperator{\AW}{AW}
\DeclareMathOperator{\EZ}{EZ}

\newcommand{\Wedge}{\textstyle\bigwedge}

\newcommand{\sig}{\sigma}
                                   
\begin{document}
\begin{abstract}
We define group-twisted 
Alexander-Whitney and Eilenberg-Zilber maps
for converting between bimodule resolutions of skew group algebras.
These algebras are the natural semidirect products
recording actions of finite groups 
by automorphisms.
The group-twisted chain maps allow
us to transfer information between resolutions for use
in homology theories,
for example, in those governing deformation theory.
We show how to translate in particular from the default 
(but often
cumbersome) reduced bar resolution
to a more convenient twisted product resolution.
This provides a more universal approach to some known results
classifying PBW deformations.
\end{abstract}

\title[Group-twisted Alexander-Whitney and Eilenberg-Zilber maps]
{  Group-twisted Alexander-Whitney and Eilenberg-Zilber maps}

\date{April 19, 2020}



\author{A.\ V.\ Shepler}
\address{Department of Mathematics, University of North Texas,
Denton, TX 76203, USA}
\email{ashepler@unt.edu}
\author{S.\  Witherspoon}
\address{Department of Mathematics\\Texas A\&M University\\
College Station, TX 77843, USA}\email{sjw@math.tamu.edu}
\thanks{The first author was partially supported by Simons grant 429539.
The second author was partially supported by NSF grant DMS-1665286.
Corresponding Author: Anne Shepler.}

\maketitle

\section{Introduction}

Skew group algebras encode actions of groups on other algebras by automorphisms.
Deformations of skew group algebras include, for example,
symplectic reflection algebras, rational Cherednik algebras,
graded Hecke algebras, and Drinfeld orbifold algebras;
e.g., see~\cite{Drinfeld,EG,Lusztig89,RamShepler}. 
These are all PBW deformations of $S(V)\rtimes G$ 
for the action of a finite group $G$ on a finite dimensional 
vector space $V$ with symmetric algebra $S(V)$.
The Hochschild cohomology of $S(V)\rtimes G$ in turn captures
important information about its deformations.
In positive characteristic, the situation is more complicated 
than over fields of characteristic 0,
as the group algebra $kG$ itself is not always semisimple.
There are a number of recent papers on the representation theory of such
deformations in positive characteristic;
see, e.g.,~\cite{BalagovicChen,BalagovicChen2,Bellamy-Martino,Brown-Changtong,Norton}. 
In this short note, we show how to understand
some prior results on PBW deformations, in positive
characteristic particularly,
more conceptually in terms of 
Alexander-Whitney and Eilenberg-Zilber
maps twisted by group actions.

We begin in Section~\ref{sec:NW1} by recalling a construction of twisted product
resolutions for skew group algebras.
In Section~\ref{sec:AW-EZ}, we define the group-twisted Alexander-Whitney
and Eilenberg-Zilber maps informing deformation theory. 
We prove in Section~\ref{sec:conversion} 
our main Theorem~\ref{thm:chain} and show that composing
with these maps provides a way to embed convenient resolutions into bar resolutions
as direct summand subcomplexes,
generalizing results from earlier papers.
In Section~\ref{sec:deformations}, we explain how some previous results in deformation
theory follow from our main theorem in a more conceptual way. 

Throughout, $k$ will be a field of arbitrary characteristic,
and $\ot = \ot_k$. 


\section{Resolutions for skew group algebras}\label{sec:NW1}

Consider a finite group $G$ acting on a $k$-algebra $S$ by automorphisms. 
The resulting skew group algebra
$S\rtimes G$ is the free $S$-module with basis $G$ and multiplication
\[
   (sg) \cdot (s'g') = s ({}^gs')\, g g'
   \quad\text{for all}\quad s,s' \in S\ \text{ and }\ g, g'\in G
\]
where $ {}^g s'$ denotes
the action of $g$ on $s'$.  
We recall a construction from~\cite{ueber,twisted}.

\subsection*{The twisted product resolution}
Consider projective resolutions
\begin{equation*}
\begin{aligned}
\text{(i)}\quad&
C_{\DOT}: \ \ \ldots \rightarrow C_2 \rightarrow C_1 \rightarrow C_0 \rightarrow 0
\ \ \  \text{ of } kG \text{ as a $kG$-bimodule, and}&
\\
\text{(ii)} \quad&
D_{\DOT}:\ \  \ldots \rightarrow D_2 \rightarrow D_1 \rightarrow D_0 \rightarrow 0\
\ \text{ of } S \text{ as an $S$-bimodule.}&
\end{aligned}
\end{equation*}
We take $C_{\DOT}$ to be $G$-graded, with group action compatible 
with the grading, that is, 
$$
g_1\big((C_i)_{g_2}\big)g_3 = (C_i)_{g_1g_2g_3}
\quad\text{ for all }\ g_1,g_2,g_3\in G
\ \text{and all degrees } i\, ,
$$ 
and choose a resolution $D_{\DOT}$ upon which $G$ acts:
Each $D_{i}$ should be a left $kG$-module 
with $g\cdot (s\cdot d) = {}^gs\cdot (g\cdot d)$
for $g$ in $G$, $s$ in $S$, and $d$ in $D$
and with differentials $kG$-module homomorphisms. 
Then  $D_{\DOT}$ is 
said to be {\em compatible} with the twisting map
given by the group action (see~\cite[Definition~2.17]{twisted}). 
For example, these conditions all hold when $C_{\bu}$ is the bar resolution
or reduced bar resolution of $kG$ and $D_{\bu}$ is the Koszul resolution
of a Koszul algebra $S$ with action of $G$ by automorphisms
(see~\cite[Prop 2.20(ii)]{twisted}).

We combine the above resolutions of $kG$ and $S$
to construct a resolution of $S \rtimes G$:
The {\em twisted product resolution} $C \ot^G D$ of the algebra $S\rtimes G$
is given as a complex of vector spaces by
the total complex of the double complex 
$C_{\DOT}\ot D_{\DOT}$\,: 
\begin{equation}\label{eqn:res-X}
(C\ot^G D)_n = \bigoplus_{i+j=n} C_i \ot D_j\, .
\end{equation}
We imbue $C \ot^G D$ with an $(S\rtimes G)$-bimodule structure on its components 
by defining an action of $S$ on the left 
and an action of $kG$ on the right given by
\[
   s \cdot (c\ot d)\cdot h 
   \ = \
   ch \ot {}^{(gh)^{-1}}s \cdot {}^{h^{-1}}d
\quad\text{ for }\ \ 
g,h\in G,\ c\in C_g,\ d\in D,\ s\in S\, .
\]
Then the complex $C \ot^G D$, augmented by $S\rtimes G$,
is indeed an exact sequence of $(S\rtimes G)$-bimodules
(see~\cite{twisted} or ~\cite[\S4]{quad}).
It is a projective resolution of $S\rtimes G$ when
each $(C\ot^G D)_n$ is projective 
as a $(S\rtimes G)$-bimodule.
This is the case, for example, when $D$ is a Koszul resolution
of a Koszul algebra and $C$ is the bar resolution of $kG$.
In this case, each $D_j$ is a free $S$-module and each $C_i$ is a free
$kG$-module, and it can be shown directly that $C_i\ot D_j$
is a free $(S\rtimes G)$-bimodule 
(with basis given by tensoring the basis elements of $C_i$ with those
of $D_j$).

\subsection*{Reduced bar resolutions}
We now consider a special case of the twisted product resolution.
Fix $C = \overline{B}_G$ and $D=\overline{B}_S$,
the reduced bar resolutions for $G$ and $S$, respectively:
$$
(\overline{B}_G)_n = kG \ot \overline{kG}^{\, \ot n}\ot kG
\quad\text{ and }\quad
(\overline{B}_S)_n = S\ot \overline{S}^{\, \ot n}\ot S
$$ where 
$\overline{kG}=kG/(k\cdot 1_G)$ as a vector space, and similarly 
$\overline{S}=S/(k\cdot 1_S)$.
We fix some choice of section $\overline{S}\hookrightarrow S$
and choose the embedding
$G-\{1\} \hookrightarrow G$ to define a section $\overline{kG}\hookrightarrow kG$.
There is a twisted product resolution (see~\cite[Proposition 2.20(ii) and Corollary~3.12]{twisted})
of $S \rtimes G$ formed by these reduced bar complexes:
$$
C\ot^G D = \overline{B}_G\ot^G \, \overline{B}_S\, .
$$  
It is given as a vector space in degree $n$ by
$$
(\overline{B}_G\ot^G \, \overline{B}_S)_n
=\bigoplus_{\ell=0}^{n}\ 
 kG\ot (\overline{kG})^{\,\ot (n-\ell)}\ot kG\ot
        S\ot \overline{S}^{\,\otimes \ell}\ot S .
$$

We compare this twisted product resolution of $S\rtimes G$
with its reduced bar resolution $\overline{B}_{S\rtimes G}$.
Here,
$\overline{(S\rtimes G)}= (S\rtimes G)/ (k\cdot 1_{S\rtimes G})$
with $1_{S\rtimes G}=1_S\cdot 1_G$ and 
$$ 
(\overline{B}_{S\rtimes G})_n = 
(S\rtimes G)
\ot 
\overline{(S\rtimes G)}^{\, \otimes n}
\ot 
(S\rtimes G)\, .
$$


\section{Group-twisted Alexander-Whitney and Eilenberg-Zilber maps}\label{sec:AW-EZ}

We now define analogs of the Alexander-Whitney
and Eilenberg-Zilber maps in the environment of group actions.
(Compare with~\cite{GG,LZ}
for the case of tensor products of algebras
and~\cite{GNW} for twisted tensor products
defined by bicharacters on grading groups.) 

Recall that the traditional Alexander-Whitney
and Eilenberg-Zilber maps convert between the reduced bar resolution
of a tensor product and the tensor product of the individual reduced bar resolutions:$$
\begin{aligned}
  \AW_n & : & (\overline{B}_{S\ot kG})_n  \ \longrightarrow&   \ 
  (\overline{B}_G\ot \, \overline{B}_S)_n
  , \\
  \EZ_n & : &  
  (\overline{B}_S\ot \, \overline{B}_G)_n
    \ \longrightarrow&  \ 
          (\overline{B}_{S\ot kG})_n .
\end{aligned} 
$$
%
We define analogs of these maps for
the skew group algebra $S\rtimes G$.  This algebra is $S\ot kG$ as a 
vector space but with multiplication
given by a twisting map $kG \otimes S \rightarrow S \ot kG$.
Consequently, we switch the order of $kG$ and $S$
when taking the tensor products of the individual bar resolutions.
We define maps, for each $n$, 
$$
\begin{aligned}
  \AW^G_n & : & (\overline{B}_{S\rtimes G})_n  \ \longrightarrow&   \ 
  (\overline{B}_G\ot^G \, \overline{B}_S)_n
  , \\
  \EZ^G_n & : &  
  (\overline{B}_G\ot^G \, \overline{B}_S)_n
    \ \longrightarrow&  \ 
          (\overline{B}_{S\rtimes G})_n .
\end{aligned} 
$$
For ease of notation, we write elements of $\overline{kG}$ and $\overline{S}$
just as if they were in $kG$ and $S$,
respectively,  invoking our choices of section maps,
and no confusion should arise. 

\subsection*{Group-twisted Alexander-Whitney map}
Define the first map by 
\[
\begin{aligned}
    \AW^G_n & (1\ot   s_1g_1 \ot\cdots\ot s_ng_n \ot 1 ) \\
  &  \hspace{3em}  
  =  
  \sum_{\ell=0}^n (-1)^{\ell(n-\ell)} (g_1\cdots 
   g_\ell) \ot g_{\ell+1}\ot\cdots \ot g_n\ot 1\, \ot
     \\
&\hspace{7em}
  1\ot{}^{(g_1\cdots g_n)^{-1}}s_1\ot\cdots\ot   {}^{(g_{\ell}\cdots g_n)^{-1}}s_\ell
   \ot  {}^{(g_{\ell+1}\cdots g_n)^{-1}}(s_{\ell+1} \cdots s_n)
\end{aligned}
\]
for $g_1,\ldots, g_n$ in $G$ and $s_1,\ldots,s_n$ in $S$. 
Here,
any tensor component in $k\cdot 1_G$ (respectively,  
$k\cdot 1_S$ or $k\cdot 1_{S\rtimes G}$)
is understood to be zero as an element of $\overline{kG}$
(respectively, $\overline{S}$ or $\overline{S\rtimes G}$). 
For each $n$, this defines the {\em group-twisted Alexander-Whitney map}
$\AW^G_n$ as an $(S\rtimes G)$-bimodule homomorphism,
since it is defined on a free basis. 

\subsection*{Group-twisted Eilenberg-Zilber map}
We shuffle the elements of $G$ with the elements of $S$ to define the second map.  Let $\mathfrak{S}_{n-\ell,\ell}$ be the set of {\em $(n-\ell,\ell)$-shuffles} in the symmetric group
$\mathfrak{S}_n$, i.e., permutations $\sigma$ of $1, \ldots, n$ with
\[
      \sigma(1)<\sigma(2)<\cdots <\sigma(n-\ell) \ \ \mbox{ and } \ \ 
      \sigma(n-\ell+1)<\sigma(n-\ell+2)<\cdots < \sigma(n) .
\]
We also count the number of {\em inversions} of a permutation
$\sigma$ in $\mathfrak{S}_n$: Set
\[
     |\sig| = | \{ (i,j) \mid 1\leq i<j\leq n \ \mbox{ and } \ \sig(i)>\sig(j)\}|\ .
\]
For each shuffle $\sig^{-1} \in \mathfrak{S}_{n-\ell,\ell}$
and tuple $(g_1,\ldots, g_{n-\ell}, s_1,\ldots, s_{\ell}) 
=(x_1,\ldots, x_{n})$, define
\[ 
    F_{\sig}(x_1\ot\cdots\ot x_n) =\
    ^{h_{\sig(1)}}(x_{\sig(1)})\ot\cdots\ot\, ^{h_{\sig(n)}}(x_{\sig(n)})
\]
where $h_i = 1_G$
for $1\leq i \leq n-\ell$ (so $\, ^{h_i}x_i = x_i$ for $x_i \in G$) and 
$$
h_i=\prod_{\substack{ \sig^{-1}(i)+1\,\leq\, j\,\leq n 
   \\ \sig(j)\, \leq\, n-\ell\rule{0ex}{1.5ex}}} 
   g_{\sig(j)}\, \text{ in } G
 \quad\text{ for } \ 
n-\ell+1\leq i\leq n \quad\text{(when $x_i \in S$)}.
$$
We define the {\em group-twisted Eilenberg-Zilber map} by
\[
\begin{aligned} 
   \EZ^G_n & (1\ot g_1\ot\cdots\ot g_{n-\ell}\ot 1\ot 1\ot s_1\ot\cdots\ot s_\ell\ot 1 )\\
 &\hspace{2em} =
   \sum_{\sig^{-1}\in \mathfrak{S}_{n-\ell,\ell}} (-1)^{|\sig|} \ot F_{\sig} (g_1\ot \cdots \ot g_{n-\ell}\ot
    s_1\ot\cdots\ot s_\ell) \ot 1
\end{aligned}
\] 
for all $g_1,\ldots, g_{n-\ell}$ in $G - \{1\}$ and 
$s_1,\ldots, s_\ell$ in $\overline{S}$.
For each $n$, this defines $\EZ^G_n$ as an $(S\rtimes G)$-bimodule homomorphism,
since it is defined  on a free basis. 

\subsection*{Group-twisted maps are chain maps}
These group-twisted maps convert between two resolutions of the algebra $S\rtimes G$:
the reduced bar resolution $\overline{B}_{S\rtimes G}$ on one hand
and the twisted product resolution $\overline{B}_{G}\otimes^{G}\,\overline{B}_S $
on the other hand.
\begin{lemma}
\label{lem:chainmaps}
The group-twisted maps  $\AW^G_n$ and $\EZ^G_n$ are chain maps:
$$
\entrymodifiers={+!!<0pt,\fontdimen22\textfont2>}
\xymatrixcolsep{6ex}
\xymatrixrowsep{8ex}
\xymatrix{
(\overline{B}_{S\rtimes G})_{n+1}   
\ar[rr]^{d_{n+1}}  \ar@<1ex>[d]^{\rm{AW}^G}
 &  & 
 (\overline{B}_{S\rtimes G})_{n}
\ar@<1ex>[d]^{\rm{AW}^G}
&  & 
 \\
(\overline{B}_{G} \otimes^{G} \overline{B}_S )_{n+1} 
 \ar[rr]^{d_{n+1}}  \ar[u]^{\rm{EZ}^G}  
 & & 
(\overline{B}_{G} \otimes^{G} \overline{B}_S )_{n} 
\ar[u]^{\rm{EZ}^G} 
 & &
}
$$
\end{lemma}
\begin{proof}
A calculation shows that $\AW^G$ is a chain map just as in the case of 
the traditional Alexander-Whitney map for the tensor product of algebras
(see, e.g.,~\cite[(X.7.2)]{MacLane}).
For example, in degree~2, 
\begin{small}
\[
\begin{aligned}
d_2 \AW^G_2 & (1\ot s_1 g_1\ot s_2g_2\ot 1) \\
&= d_2 ((1\ot g_1\ot g_2\ot 1)\ot (1\ot {}^{g_1^{-1}g_2^{-1}}s_1{}^{g_2^{-1}}s_2) 
- (g_1\ot g_2\ot 1)
  \ot(1\ot {}^{g_1^{-1}g_2^{-1}}s_1\ot {}^{g_2^{-1}}s_2) \\
& \hspace{2cm}+ (g_1g_2\ot 1)\ot 
  (1\ot {}^{g_1^{-1}g_2^{-1}}s_1\ot {}^{g_2^{-1}}s_2\ot 1))\\
&= 
(g_1\ot g_2\ot 1)\ot (1\ot {}^{g_1^{-1}g_2^{-1}}s_1{}^{g_2^{-1}}s_2) - 
  (1\ot g_1g_2\ot 1)\ot (1\ot {}^{g_1^{-1}g_2^{-1}}s_1{}^{g_2^{-1}}s_2) 
  \\
  & \hspace{1cm}
   + (1\ot g_1\ot g_2)\ot
    (1\ot {}^{g_1^{-1}g_2^{-1}}s_1{}^{g_2^{-1}}s_2)
  - (g_1g_2\ot 1)\ot (1\ot {}^{g_1^{-1}g_2^{-1}}s_1\ot {}^{g_2^{-1}}s_2) \\
   & \hspace{1cm}
  + (g_1\ot g_2)\ot (1\ot {}^{g_1^{-1}g_2^{-1}}s_1
  \ot {}^{g_2^{-1}}s_2) 
+ (g_1\ot g_2\ot 1)\ot ({}^{g_1^{-1}g_2^{-1}}s_1\ot {}^{g_2^{-1}}s_2) \\
 & \hspace{1cm}
 - (g_1\ot g_2\ot 1)\ot (1\ot
   {}^{g_1^{-1}g_2^{-1}}s_1 {}^{g_2^{-1}}s_2)
  +
(g_1g_2\ot 1)\ot ( {}^{g_1^{-1}g_2^{-1}}s_1\ot {}^{g_2^{-1}}s_2\ot 1)\\
 & \hspace{1cm}
  - (g_1g_2\ot 1)\ot (1\ot {}^{g_1^{-1}g_2^{-1}}s_1{}^{g_2^{-1}}s_2\ot 1)
  + (g_1g_2\ot 1)\ot (1\ot {}^{g_1^{-1}g_2^{-1}}s_1\ot {}^{g_2^{-1}}s_2) ,
\end{aligned}
\]
\end{small}
whereas
\begin{small}
\[
\begin{aligned}
\AW_1^G d_2 &(1\ot s_1g_1\ot s_2g_2\ot 1) \\
 &= \AW_1(s_1g_1\ot s_2g_2\ot 1 - 1\ot s_1 {}^{g_1}s_2 g_1g_2\ot 1 + 1\ot s_1g_1
    \ot s_2g_2 ) \\
  &= s_1 g_1((1\ot g_2\ot 1) \ot (1\ot {}^{g_2^{-1}}s_2) + (g_2\ot 1)\ot 
(1\ot {}^{g_2^{-1}}s_2\ot 1))\\
 & \hspace{1cm}- (1\ot g_1g_2\ot 1) \ot (1\ot {}^{g_1^{-1}g_2^{-1}}s_1{}^{g_2^{-1}}s_2) 
- (g_1g_2\ot 1)\ot  (1\ot {}^{g_1^{-1}g_2^{-1}}s_1{}^{g_2^{-1}}s_2\ot 1)\\
 & \hspace{1cm}+ ((1\ot g_1\ot 1)\ot (1\ot {}^{g_1^{-1}}s_1) + 
 (g_1\ot 1)\ot (1\ot {}^{g_1^{-1}}s_1\ot 1))s_2g_2\\
& = (g_1\ot g_2\ot 1 ) \ot ({}^{g_1^{-1}g_2^{-1}}s_1\ot {}^{g_2^{-1}}s_2) 
   + (g_1g_2\ot 1)\ot ({}^{g_1^{-1}g_2^{-1}}s_1\ot {}^{g_2^{-1}}s_2\ot 1)\\
    &\hspace{1cm}- (1\ot g_1g_2\ot 1) \ot (1\ot {}^{g_1^{-1}g_2^{-1}}s_1{}^{g_2^{-1}}s_2) 
  - (g_1g_2\ot 1)\ot (1\ot {}^{g_1^{-1}g_2^{-1}}s_1 {}^{g_2^{-1}}s_2\ot 1)\\
 & \hspace{1cm}+ (1\ot g_1\ot g_2) \ot (1\ot {}^{g_1^{-1}g_2^{-1}}s_1{}^{g_2^{-1}}s_2) 
   + (g_1\ot g_2)\ot (1\ot {}^{g_1^{-1}g_2^{-1}}s_1\ot {}^{g_2^{-1}}s_2) , 
\end{aligned}
\]
\end{small}and we see that these two expressions are equal after some cancellations.
Similarly, to show $\EZ^G$ is a chain map, we follow the
proof that the classical Eilenberg-Zilber map $\EZ$ is a chain map but include
the group action.
(See, e.g.,~\cite[(VIII.8.9)]{MacLane} and~\cite{LZ}.)
\end{proof}

Next we see that the maps $\AW^G$, $\EZ^G$
provide a splitting of the reduced
bar resolution of $S\rtimes G$, 
with a copy of the resolution $\overline{B}_G\ot ^G \overline{B}_S$
as a direct summand. 

\begin{lemma}\label{lem:AW-EZ}
$ \ \AW^G \EZ^G$ is the identity map in each degree.
\end{lemma}

\begin{proof}
If the action of $G$ is trivial, we have the standard 
Alexander-Whitney and Eilenberg-Zilber maps, and
$\AW^1 \EZ^1$ is known to be the identity map~\cite[Remark 3.2]{LZ}. 
One can verify that $\AW^G \EZ^G$ is the identity
map for nontrivial group actions as well, via a straightforward
but tedious calculation, keeping careful track of group actions. 
For example, in degree~2, for all $g$ in $G$ and $s$ in $S$, 
\begin{eqnarray*}
\AW^G_2\EZ^G_2 ((1\ot g\ot 1)\ot (1\ot s\ot 1))
  &=& \AW^G_2 ( 1\ot g\ot s\ot 1 - 1\ot {}^gs\ot g\ot 1)\\
&=& (1\ot g\ot 1)\ot (1\ot s\ot 1) .
\end{eqnarray*}
Note that many terms in the image of $\AW^G_2$ are 0 since we
are working with the {\em reduced} (instead of unreduced) bar resolution. 
\end{proof}

\section{Conversion between resolutions}\label{sec:conversion}
In this section, we show that the group-twisted 
Alexander-Whitney and Eilenberg-Zilber maps
provide a way to convert between the reduced bar resolution
for a skew group algebra $S\rtimes G$
and a choice of twisted product resolution for $S\rtimes G$,
which generally can be much smaller. 
Compare to~\cite[Lemma~4.4]{SW-Koszul} for Koszul algebras;
our result here is stated for more general algebras $S$, and we provide
a proof based on the explicit chain maps $\AW^G$
and $\EZ^G$ defined above. 
We will see in the next section that
the maps $\AW^G$ and $\EZ^G$
have deformation-theoretic consequences. 

Again, consider compatible projective resolutions 
\begin{equation*}
\begin{aligned}
\text{(i)}\quad&
C_{\DOT}: \ \ \ldots \rightarrow C_2 \rightarrow C_1 \rightarrow C_0 \rightarrow 0
\ \ \  \text{ of } kG \text{ as a $kG$-bimodule, and}&
\\
\text{(ii)} \quad&
D_{\DOT}:\ \  \ldots \rightarrow D_2 \rightarrow D_1 \rightarrow D_0 \rightarrow 0\
\ \text{ of } S \text{ as an $S$-bimodule}&
\end{aligned}
\end{equation*}
as in Section~\ref{sec:NW1} and their
twisted product resolution $C\ot^G D$. 
We combine the group-twisted Alexander-Whitney
and Eilenberg-Zilber maps with chain maps converting 
between $C$, $D$ and
the reduced bar resolutions
$\overline{B}_G$, $\overline{B}_S$,
respectively, to construct a chain map
splitting
\vspace{3ex}
\begin{equation*}
\label{splitting}
\xymatrix{
  &  C\ot^G D
  \ar[rr] &
  & \overline{B}_{S\rtimes G}\ .
  \ar@(ul,ur)@{-->}[ll] \, 
}
\end{equation*}
We take the setting when chain maps 
$$
\begin{aligned}
&\iota_G : C\rightarrow \overline{B}_G,\quad 
&\pi_G :\overline{B}_G\rightarrow C,
\quad\quad
&\iota_S : D\rightarrow \overline{B}_S,\quad
&\pi_S: \overline{B}_S\rightarrow D
\end{aligned}
$$ 
are given for which
$\iota_S, \pi_S$ are $kG$-module homomorphisms
and $\pi_G\iota_G$ and $\pi_S\iota_S$ are identity maps in each degree:
$$
\hspace{5ex}
\mbox{
\entrymodifiers={+!!<0pt,\fontdimen22\textfont2>}
\xymatrixcolsep{6ex}
\xymatrixrowsep{6ex}
\xymatrix{
(\overline{B}_G )_{n}
\ar[rr]^{d_{n+1}}  \ar@<1ex>[d]^{\pi_G}
 &  & 
( \overline{B}_G )_{n+1}
\ar@<1ex>[d]^{\pi_G}
&  & 
 \\
C_{n+1} 
 \ar[rr]^{d_{n+1}}  \ar[u]^{\iota_G}  
 & & 
\ C_{n} 
\ar[u]^{\iota_G} 
 & &
}
}
\hspace{-7ex}
\mbox{
\entrymodifiers={+!!<0pt,\fontdimen22\textfont2>}
\xymatrixcolsep{6ex}
\xymatrixrowsep{6ex}
\xymatrix{
(\overline{B}_{S} )_{n}
\ar[rr]^{d_{n+1}}  \ar@<1ex>[d]^{\pi_S}
 &  & 
(\overline{B}_{S} )_{n+1}
\ar@<1ex>[d]^{\pi_S}
&  & 
 \\
D_{n+1} 
 \ar[rr]^{d_{n+1}}  \ar[u]^{\iota_S}  
 & & 
\ D_{n}
\ar[u]^{\iota_S} 
 . &  &
}
}
$$
For example, if $S=S(V)$ is a symmetric algebra on a finite dimensional
vector space $V$ with an action of $G$ by automorphisms, and $D$
is the Koszul resolution of $S$,
maps $\iota_S$, $\pi_S$ are constructed explicitly
in~\cite[(2.6) and (4.2)]{SW-quantum},
and the hypotheses of the next theorem hold. 

\begin{thm}\label{thm:chain}
Let $G$ be a finite group acting on a $k$-algebra $S$ by automorphisms.
Let $\overline{B}_{S\rtimes G}$ be the reduced bar resolution for
$S\rtimes G$. 
Let $C \ot ^G D$
be a twisted product resolution for
$S \rtimes G$ obtained by twisting together a resolution
$C_{\DOT}$ of $kG$ and a resolution $D_{\DOT}$ of $S$ as above.
There exist chain maps 
$$\iota: (C \ot ^G D)_{\bu}\rightarrow 
(\overline{B}_{S\rtimes G})_{\bu}
\quad\text{ and }\quad
\pi:(\overline{B}_{S\rtimes G})_{\bu}   
\rightarrow (C \ot ^G D)_{\bu}
$$ 
for which $\pi_n \iota _n: (C \ot ^G D)_n \rightarrow (C \ot ^G D)_n$ is the identity map for all $n$.
If $S$ is a graded algebra and $G$ consists of elements of degree~0,
then the chain maps $\iota$, $\pi$ are graded maps provided 
$\iota_S$ and $\pi_S$ are graded maps. 
\end{thm}
\begin{proof}
Define compositions
\begin{equation}\label{eqn:iota-pi}
\iota = \EZ^G (\iota_G\ot \iota_S)
     \ \ \ \mbox{ and } \ \ \ 
   \pi= (\pi_G\ot \pi_S) \AW^G , 
\end{equation}
which can be visualized via a diagram: 
$$
\entrymodifiers={+!!<0pt,\fontdimen22\textfont2>}
\xymatrixcolsep{2ex}
\xymatrixrowsep{8ex}
\xymatrix{
& &
(\overline{B}_{S\rtimes G})_{n+1}   
\ar[rr]  \ar@<1ex>[d]^{\rm{AW}^G}
&  & 
 (\overline{B}_{S\rtimes G})_{n}
\ar@<1ex>[d]^{\rm{AW}^G}
& \ar@(dr, ur)[dd] ^{\pi} & 
 \\
& & 
 (\overline{B}_{G} \otimes^{G} \overline{B}_S )_{n+1} 
 \ar[rr]  \ar[u]^{\rm{EZ}^G}
 \ar@<1ex>[d]^{\pi_G \ot \pi_S} 
 & & 
(\overline{B}_{G} \otimes^{G} \overline{B}_S )_{n} 
\ar[u]^{\rm{EZ}^G}  
\ar@<1ex>[d]^{\pi_G \ot \pi_S}  
&  & 
  \\ 
& \ar@(ul, dl)[uu]^{\iota}
&
(C \otimes^{G} D )_{n+1} 
 \ar[rr]  \ar[u]^{\iota_G \ot\,  \iota_S}
 & & 
(C \otimes^{G} D )_{n} 
\ar[u]^{\iota_G \ot\,  \iota_S}
  & &
}
$$
By Lemmas~\ref{lem:chainmaps} 
and~\ref{lem:AW-EZ} and the hypotheses on $\iota_G,\iota_S,\pi_G,\pi_S$,
the composition $\pi\iota$ is $1$.
\end{proof} 

\section{Applications to deformation theory}\label{sec:deformations}

We now consider some applications of Theorem~\ref{thm:chain}.
We used Hochschild cohomology 
to analyze
PBW deformations of $S\rtimes G$ 
in~\cite{ueber} and~\cite{SW-Koszul},
particularly in the case that $S=S(V)$
is a symmetric algebra.
Our Theorem~\ref{thm:chain} generalizes both~\cite[Lemma~7.3]{ueber}
and~\cite[Lemma~4.4]{SW-Koszul}, which
provide an explicit connection between
Hochschild cocycles and PBW deformations.
We explain here how Theorem~\ref{thm:chain}
gives a unified proof of this connection using
the group-twisted Alexander-Whitney and Eilenberg-Zilber maps.
This approach is more elegant than the somewhat ad hoc proofs given in~\cite{ueber,SW-Koszul}.  It also more easily applies to other settings.
We emphasize that we did not assume $S$ is Koszul in
Theorem~\ref{thm:chain}
in the last section, although $S$ will be a Koszul algebra
here.

Let $V$ be a finite dimensional $kG$-module and $S=S(V)$, the 
symmetric algebra on the underlying vector space $V$.
Consider some linear parameter functions
$$
\kappa:V\ot V \rightarrow kG \ \ \ \mbox{ and } \ \ \ 
\ld:kG\ot    V \rightarrow kG\, 
$$
with $\kappa$ alternating.
We view $\kappa$ as a function on the exterior power $\Wedge^2V$.

Let
$\cH_{\ld,\kappa} $ 
be the associative $k$-algebra generated by a basis of $V$ 
and the group algebra $kG$ with relations given by those of $kG$ 
together with
\begin{equation*}
\begin{aligned}
vw-wv=\kappa(v\wedge w)\quad\text{ and }\quad
gv- \,  ^g\! v g = \ld(g\otimes v)
\quad\text{ for }v,w\in V,\ g \in G.
\end{aligned}
\end{equation*}
Then $\cH_{\ld,\kappa}$ 
is a filtered algebra with $\deg(v) =1$ and $\deg(g)=0$
for all $v$ in $V$ and $g$ in $G$.
We call $\cH_{\lambda,\kappa}$ a {\em PBW deformation} of $S(V)\rtimes G$ if
$\gr\cH_{\ld,\kappa}\cong S(V)\rtimes G$, that is, the associated
graded algebra of $\cH_{\ld,\kappa}$ is isomorphic to the skew group
algebra.

Over fields $k$ of characteristic $0$,
the PBW deformations $\cH_{\ld,\kappa}$ with $\kappa\equiv 0$
include the affine graded Hecke algebras
defined by Lusztig~\cite{Lusztig88,Lusztig89};
those with $\lambda\equiv 0$
include the symplectic reflection algebras and Drinfeld Hecke
algebras (see~\cite{EG,Griffeth,RamShepler}).
Note that over fields of characteristic $0$
(or more generally coprime to the group order $|G|$),
every PBW algebra of the form
$\cH_{\ld,0}$ is isomorphic to one of the form
$\cH_{0,\kappa}$,
but this fact does not hold
in positive characteristic (see~\cite[Theorem~4.1, Example 5.1]{ueber}).

The PBW property imposes conditions on $\kappa,\lambda$, as given in the 
following theorem, which is~\cite[Theorem~3.1]{ueber}.
Observe that when $\lambda$ is zero,
all but conditions~(2) and~(4) in the theorem
are trivial. 

\begin{thm}\label{thm:PBWconditions}
Let $k$ be a field of arbitrary characteristic.
The algebra $\cH_{\ld,\kappa}$ is a PBW deformation
of $S(V)\rtimes G$ if and only if 
the following conditions hold
for all $g,h$ in $G$ and $u,v,w$ in $V$.
\begin{enumerate}
\item\label{cocycle21}
\, \ \rule[0ex]{0ex}{3ex}
$\ld(gh,v)=\ld(g,\, ^hv) h + g\ld(h,v)$ in $kG$.
\item\label{firstobstruction12}\ \ \rule[0ex]{0ex}{3ex}
$\kappa(\, ^gu,\, ^g v)g-g\kappa(u,v)
=
\ld\bigl(\ld(g,v),u\bigr)-\ld\bigl(\ld(g,u),v\bigr)$ in $kG$.
\item\label{cocycle12}\ \ \rule[0ex]{0ex}{3ex}
$\ld_h(g,v)(\, ^hu-\, ^gu)=\ld_h(g,u)(\, ^hv-\, ^gv)
$ in $V$.
\item\label{firstobstruction03}
\ \ \rule[0ex]{0ex}{3ex}
$\kappa_g(u,v)(^gw-w)+\kappa_g(v,w)(^gu-u)+\kappa_g(w,u)(^gv-v)=0$ in $V$.
\item\label{mixedbracket}
\ \ \rule[0ex]{0ex}{3ex}
$\ld\bigl(\kappa(u,v),w\bigr)
+\ld\bigl(\kappa(v,w),u\bigr)+\ld\bigl(\kappa(w,u),v\bigr)=0$ in $kG$.
\end{enumerate}
\end{thm}

The five conditions in Theorem~\ref{thm:PBWconditions} 
may not seem intuitive at first glance,
and the proof given in~\cite{ueber} indeed uses the Diamond Lemma
(see~\cite{Bokut} and~\cite{Bergman}).
Hochschild cohomology gives a natural interpretation for the 
five conditions in terms of Gerstenhaber brackets and
cocycles, as explained in~\cite[Section~8]{ueber}.
There, we considered 
a chain map from a twisted product resolution $X_{\bu}$ to the bar resolution,
$$\phi_{\bu}: X_{\bu} \longrightarrow B_{S\rtimes G},
$$ 
where $X_{\bu}=C\ot^G D$ is defined as in~(\ref{eqn:res-X})
for $C$ the bar resolution of $kG$ and $D$ the Koszul resolution of $S$. 
We used this chain map 
to translate the five conditions of Theorem~\ref{thm:PBWconditions} 
into provisions more natural 
in the architecture of cohomology.
But the chain map $\phi_{\bu}$ was constructed there in an improvised
and utilitarian way.
We now explain how to use Theorem~\ref{thm:chain} instead to give a direct
proof of~\cite[Theorem 8.3]{ueber}:

\begin{thm}\label{thm:PBWcohomologyconditions}
Let $k$ be a field of arbitrary characteristic.
The algebra $\cH_{\ld,\kappa}$ 
exhibits the PBW property if and only if
\begin{itemize}
\item
$d^*(\lambda)=0$,
\item
$[\lambda,\lambda]=2d^*(\kappa)$, and
\item
$[\lambda,\kappa]= 0 $
\end{itemize}
as Hochschild cochains, 
where $\lambda$ and $\kappa$ are identified with
cochains on the resolution $X_{\DOT}$.
\end{thm}
\begin{proof}
Let $\text{Kosz}_{S(V)}$ be the Koszul resolution of the polynomial ring
$S(V)$ and set
$$
S=S(V), \quad D=\text{Kosz}_{S(V)}, \quad
C=\overline{B}_{kG},
\quad\text{ and }
X=C\ot^G D\, .
$$
We take the standard embedding $\iota_S=\text{Kosz}_{S}\hookrightarrow \overline{B}_{S}$
and set $\iota_C = 1$, the identity map on $\overline{B}_{kG}$.
Then Theorem~\ref{thm:chain} implies existence of maps 
$$ 
 \iota : (C\ot ^G D)_{\bu} \rightarrow (\overline{B}_{S\rtimes G})_{\bu} 
  \ \ \ \mbox{ and } \ \ \ \pi :(\overline{B}_{S\rtimes G})_{\bu}\rightarrow
  (C\ot^G D)_{\bu} 
$$
for which $\pi\iota$ is the identity map in each degree. 
The map $\phi_{\bu}$ constructed in~\cite[Lemma 7.3]{ueber}
is essentially our map $\iota$:
There, $C$ was the unreduced bar resolution of $kG$,
whereas here we use the reduced
bar resolution instead 
so that $\AW^G \EZ^G=1$ 
in general,
avoiding the need for~\cite[Lemma 7.3]{ueber}.

We identify
the parameter
$\lambda$ of the algebra $H_{\lambda, \kappa}$
with a cochain in 
$$\Hom_{(S\rtimes G)^e}(X_{1,1}, S\rtimes G),
$$ 
in the following way:
Since $\lambda$ factors through the quotient 
$kG\ot V\rightarrow \overline{kG}\ot V$,
we may view $\lambda$
as a function $X_{1,1} \rightarrow kG$;
composing with the chain map $\pi$
from Theorem~\ref{thm:chain}
gives a cochain $\mu_1=\pi\otimes\lambda$
in
$\Hom_{(S\rtimes G)^e}(\overline{B}_{S\rtimes G} , S\rtimes G)$ 
for which
\[
   \mu_1 (1\ot g\ot v\ot 1) - \mu_1(1\ot {}^g v\ot g\ot 1)
   = \lambda(g\ot v)
\quad\text{ for } g\in G, v\in V. \]
Here, one recycles
part of the calculation from the proof of Lemma~\ref{lem:AW-EZ}.
Note that $(S\rtimes G)^e$ denotes the enveloping algebra $(S\rtimes G)\ot (S\rtimes G)^{\text{op}}$.

Similarly, we identify $\kappa$ with a cochain $\mu_2$ in 
$\Hom_{(S\rtimes G)^e} (X_{0,2} , S\rtimes G)$
by composing again with $\pi$;
specifically, $\mu_2$ in
$\Hom_{(S\rtimes G)^e}(\overline{B}_{S\rtimes G} , S\rtimes G)$
is the cochain for which
\[
   \kappa (v\wedge w) = \mu_2(1\ot v\ot w\ot 1) - \mu_2(1\ot w\ot v\ot 1)
   \quad\text{ for } v,w\in V. 
   \]

Deformation theory
gives Hochschild conditions necessary for a cochain on the bar resolution
of an algebra to define a deformation of that algebra 
(see~\cite{GerstenhaberSchack}).
We apply these conditions
specifically to the cochains $\mu_1,\mu_2$ 
and then translate via $\iota$ to conditions on $\lambda,\kappa$.
We find that 
Conditions~(1) and~(3) of Theorem~\ref{thm:PBWconditions}
hold exactly when $\lambda$ is a cocycle
as evaluated on $X_{2,1}$ and on $X_{1,2}$, i.e.,
$d^*(\lambda)=0$.
Conditions~(2) and~(4) hold exactly when
the Gerstenhaber bracket $[\lambda,\lambda]$
coincides with $2d^*(\kappa)$
as evaluated on $X_{1,2}$ and on $X_{0,3}$.
Condition~(5) holds exactly when
the Gerstenhaber bracket $[\lambda,\kappa]$ vanishes.
Details are in the proof of~\cite[Theorem~8.3]{ueber}.
\end{proof}

\begin{remark}{\em 
While full details of the proof of Theorem~\ref{thm:PBWcohomologyconditions}
are in~\cite{ueber}, we emphasize that 
our contribution here is to show that the needed chain maps arise
automatically
from group-twisted Alexander-Whitney and
Eilenberg-Zilber maps.
In this way, these maps provide a general method
for converting from Diamond Lemma conditions to homological conditions, giving a more unified and less ad hoc approach
toward studying deformations.
We expect that this approach will lend itself
to interesting generalizations.
} 
\end{remark}


\begin{remark}{\em 
Similarly, in~\cite{SW-Koszul}, we 
consider the case when $S$ is a quadratic algebra,
and especially the case when $S$ is Koszul.
Again, we used an explicit chain map established
with improvised methods 
(see~\cite[Lemmas~4.4 and~4.6]{SW-Koszul})
to convert between resolutions 
and study explicit conditions giving PBW deformations
of $S\rtimes G$.
Indeed, the proof of~\cite[Theorem~2.5]{SW-Koszul} relies on
explicit information provided by chain maps.
Here again, chain maps may be taken from 
Theorem~\ref{thm:chain} instead to give a more
conceptual view of the deformation-theoretic results in~\cite{SW-Koszul}. 
}
\end{remark}

\end{document}